\documentclass[a4paper,twoside]{article}
\usepackage{latexsym,bm}
\usepackage{mathrsfs,amsmath,amssymb,amsthm}
\usepackage[dvipdfm]{graphicx} \usepackage{color}
\usepackage[all]{xy}
\usepackage[T1]{fontenc}
\usepackage{jmsj2009}
\newtheorem{theorem}{Theorem}[section]

\newtheorem{corollary}[theorem]{Corollary}

\newtheorem{assumption}[theorem]{Assumption}
\newtheorem{lemma}[theorem]{Lemma}

\theoremstyle{definition}

\theoremstyle{remark}

\renewcommand{\theequation}{\thetheorem}
   \newenvironment{eq}{\begin{equation}}{\end{equation}}
   \renewcommand{\theequation}{{\arabic{section}.\arabic{subsection}.\arabic{equation}}}

\newcommand{\cO}{{\mathcal O}}

\newcommand{\cC}{{\mathcal C}}

\newcommand{\cL}{{\mathcal L}}
\newcommand{\cU}{{\mathcal U}}

\newcommand{\cR}{{\mathcal R}}

\newcommand{\cS}{{\mathcal S}}
\newcommand{\cZ}{{\mathcal Z}}

\newcommand{\cF}{{\mathcal F}}
\newcommand{\cA}{{\mathcal A}}

\newcommand{\bQ}{\mathbb Q}
\newcommand{\bZ}{\mathbb Z}
\newcommand{\bR}{\mathbb R}
\newcommand{\bP}{\mathbb P}

\newcommand{\bA}{\mathbb A}

\newcommand{\bC}{\mathbb C}

\newcommand{\bT}{\mathbb T}

\newcommand{\ra}{\rightarrow}
\newcommand{\lra}{\longrightarrow}
\newcommand{\iso}{\simeq}

\newcommand{\rank}{{\rm rank}}

\renewcommand{\H}{{\mathrm H}}
\def\blfootnote{\xdef\@thefnmark{}\@footnotetext}

\DeclareMathOperator{\Spec}{Spec}

\DeclareMathOperator{\Eff}{\overline{Eff}}

\DeclareMathOperator{\Pic}{Pic}
\DeclareMathOperator{\CDiv}{CDiv}

\DeclareMathOperator{\Cl}{Cl}

\DeclareMathOperator{\Hom}{Hom}
\DeclareMathOperator{\Ext}{Ext}

\DeclareMathOperator{\pd}{div}

\begin{document}


\title{Differentials of Cox rings:\\ Jaczewski's
  theorem revisited}

\author{Oskar \textsc{K{\k{e}}dzierski} and Jaros\l{}aw
  A. \textsc{Wi\'sniewski}}

\thanks{Dedicated to the memory of
  Krzysztof Jaczewski (1955---1994)}

\thanks{Supported by Polish MNiSzW, grant N N201 420639. We thank
  J\"urgen Hausen, Andreas H\"oring and Yuri Prokhorov for discussions
  and remarks.  We are also greatly indebt to an anonymous referee who
  pointed out mistakes in an earlier version of this paper.}


\classification{Primary 14M25, 13N05, 14E30; Secondary 14C20,14F10}

\keyword{Cox ring, Mori Dream Space, Euler sequence, differentials,
  toric varieties}

\label{startpage}

\maketitle
\abstract{A generalized Euler sequence over a complete normal variety
  $X$ is the unique extension of the trivial bundle $V \otimes \cO_X$
  by the sheaf of differentials $\Omega_X$, given by the inclusion of
  a linear space $V\subset \Ext^1_X(\cO_X,\Omega_X)$. For $\Lambda$, a
  lattice of Cartier divisors, let $\cR_\Lambda$ denote the
  corresponding sheaf associated to $V$ spanned by the first Chern
  classes of divisors in $\Lambda$. We prove that any projective,
  smooth variety on which the bundle $\cR_\Lambda$ splits into a
  direct sum of line bundles is toric. We describe the bundle
  $\cR_\Lambda$ in terms of the sheaf of differentials on the
  characteristic space of the Cox ring, provided it is finitely
  generated. Moreover, we relate the finiteness of the module of
  sections of $\cR_\Lambda$ and of the Cox ring of $\Lambda.$ }


\section{Part I, preliminaries}
\subsection{Euler sequence}\label{EulerSection}

In the present paper $X$ is a normal irreducible complete complex
variety of dimension $n$ with $\Omega_X$ denoting the sheaf of
K\"ahler differential forms on $X$. If $H=\H^1(X,\Omega_X)$ and
$H_X=H\otimes\cO_X$ then extensions of $\Omega_X$ by the trivial sheaf
$H_X$ are classified by $\Ext^1(H_X,\Omega_X)=\Hom(H,H)$. A generalized
Euler sequence defined by Jaczewski \cite[Def.~2.1]{Jaczewski} is the
following short exact sequence of sheaves on $X$ which corresponds to
the class of identity in $\Hom(H,H)$:
\begin{eq}\label{EulerSequence}
\begin{array}{ccccccccc}
0&\lra&\Omega_X&\lra&\cR_X&\lra&H_X&\lra&0
\end{array}
\end{eq}%
The dual of the sheaf $\cR_X$ is called by Jaczewski {\em the
  potential sheaf} of $X$.

More generally, given a non-zero linear subspace $V\subseteq
\H^1(X,\Omega_X)$, setting $V_X=V\otimes\cO_X$, we get a sequence
\begin{eq}\label{EulerSequenceGen}
\begin{array}{ccccccccc}
0&\lra&\Omega_X&\lra&\cR_V&\lra&V_X&\lra&0
\end{array}
\end{eq}%
associated to the inclusion $V\hookrightarrow \H^1(X,\Omega_X)$.

\subsection{Theorem of Jaczewski}\label{JaczewskiSection}

The following theorem was proved by Jaczewski, \cite[Thm.~3.1]{Jaczewski}
\begin{theorem}\label{Jaczewski}
  A smooth complete variety $X$ is a toric variety if and only if
  there exists an effective divisor $D=\bigcup_{\alpha=1}^r D_\alpha$
  with simple normal crossing components $D_\alpha$ such that $$\cR_X=
  \bigoplus_{\alpha=1\dots r}\cO_X(-D_\alpha)$$ The divisors
  $D_\alpha$ are then the closures of the codimension one orbits of
  the torus action associated to the rays of the fan defining $X$.
\end{theorem}

The ``only if'' part of the above theorem was proved earlier by
Batyrev and Melnikov, \cite{BatyrevMelnikov}. The proof of ``if'' part
given by Jaczewski involved analysing log-differentials $\Omega_X(\log
D)$ and reconstructing the torus action on $X$. A part of the argument
is recovering the Lie algebra of the torus from the sequence dual to
\eqref{EulerSequence}.


It has been pointed to us by Yuri Prokhorov that Jaczewski's theorem
is related to a conjecture of Shokurov, \cite{Shokurov}, and results
of McKernan \cite{McKernan} and Prokhorov \cite{Prokhorov1},
\cite{Prokhorov2}, about characterization of toric varieties: see
e.g.~\cite[Conj.~1.1]{Prokhorov1} for a formulation of the problem.


\subsection{Atiyah extension and jet bundle }\label{JetSection}
A particular case of the sequence \eqref{EulerSequenceGen} is known as
the Atiyah extension and it is defined in \cite{Atiyah}. Namely, let
us consider a Cartier divisor $D$ on $X$, defined on a covering
$\cU=(U_i)_{i\in I}$ of $X$ by non-zero rational functions
$f_i\in\bC(X)^*$ satisfying $\pd(f_i)_{|U_i}=D_{|U_i}$. In other words
$D$ is defined by a \v{C}ech cochain
$(f_i)\in\cC^0(\cU,\bC(X)^*)$. Its \v{C}ech boundary
$g_{ij}=f_j/f_i\in\cO^*_X(U_i\cap U_j)$, which is a \v{C}ech cocycle
in $\cZ^1(\cU,\cO^*_X)$, determines the associated invertible sheaf, or
line bundle, $\cO_X(D)$ (note that the order of indices in our
definition of $g_{ij}$ may differ from the one for transition
functions in some standard textbooks).

The log-derivatives $d\log g_{ij}=dg_{ij}/g_{ij}\in\Omega_X(U_i\cap
U_j)$ form a cocycle in $\cZ^1(\cU,\Omega_X)$ which defines (up to the
constant coefficient, which we will ignore since it does not change
any of our computations) the first Chern class
$c_1(D)\in \H^1(X,\Omega_X)$ 
and also determines the following extension, \cite[Prop.~12,
Thm.~5]{Atiyah},
\begin{eq}\label{JetSequence}
\begin{array}{ccccccccc}
0&\lra&\Omega_X&\lra&\cR_D&\lra&\cO_X&\lra&0
\end{array}
\end{eq}%
The twisted sheaf $\cR_D\otimes\cO_X(D)$ is well known to be
isomorphic to the sheaf of first jets of sections of the line bundle
$\cO_X(D)$, \cite[Sect.~4 and 5]{Atiyah}. Indeed, let $h$ be a
rational function such that $s_i=hf_i$ is regular on $U_i$, for every
$i\in I$, hence it determines a section of $\cO_X(D)$. Then, because
$s_j=g_{ij}s_i$, we have the relation
$g_{ij}ds_i+s_idg_{ij}=ds_j$. Thus the pair $(ds_i, s_i)$ is the
associated jet section of $\cR_D\otimes\cO_X(D)$.
\par
By the construction, the twisted sequence \eqref{JetSequence}, which now
is as follows:
\begin{eq}\label{JetSequence1}
\begin{array}{ccccccccc}
0&\lra&\Omega_X\otimes\cO(D)&\lra&\cR_D\otimes\cO(D)&\lra&\cO(D)&\lra&0
\end{array}
\end{eq}%
splits as a sequence of sheaves of $\bC$-modules (but not
$\cO$-modules), see \cite[Sect.~4]{Atiyah}, and this yields the exact
sequence of global sections:
\begin{eq}\label{JetSequenceSections}
\begin{array}{ccccccccc}
0&\lra&\H^0(X,\Omega_X\otimes\cO(D))&\lra&\H^0(X,\cR_D\otimes\cO(D))&
\lra&\H^0(X,\cO(D))&\lra&0
\end{array}
\end{eq}%
\par

Jets of sections of $\cO_X(D)$ can be seen as differentials on the
total space of the dual bundle. In fact, in \cite[Sect.~2.1]{KPSW} it
is shown that the sequence \eqref{JetSequence} arises from the sequence
of sheaves of differentials associated to the projection of an
associated $\bC^*$ bundle to $X$.
This observation will be elaborated later in the second part of the
paper.

\subsection{Cox rings}\label{CoxRingSection}
From now on we assume that the variety $X$ is projective. Let $\Lambda
\subset \CDiv(X)$ be a finitely generated group of Cartier
divisors. We will assume that $\Lambda$ is free of rank $r$, so that
$\Lambda\iso\bZ^r$. The elements of $\Lambda$ will be denoted by
$\lambda$ or by $D_\lambda$, if we want to underline that they are
divisors. In fact, in the course of our arguments we will fix a
covering $\cU=(U_i)_{i\in I}$ and represent in this covering the
generators $D_\lambda$ of $\Lambda$ by rational functions
$f_i^\lambda\in\bC(X)^*$, which will imply a presentation of $\Lambda$
as a subgroup of \v{C}ech cochains $\cC^0(\cU,\bC(X)^*)$.

We assume that the first Chern class map $c_1:\Lambda\lra
\H^1(X,\Omega_X)$, defined as $\Lambda\ni D_\lambda\mapsto
c_1(\cO(D_\lambda))\in \H^1(X,\Omega_X)$, is an injection so, by abuse,
we will identify $\Lambda$ with a lattice in $\H^1(X,\Omega_X)$. We
define the following objects related to $\Lambda$:
\begin{itemize}
\item the subspace in cohomology
  $\Lambda_\bC=\Lambda\otimes\bC\subseteq \H^1(X,\Omega_X)$,
\item the sheaves of $\cO_X$-modules $\Lambda_X=\Lambda_\bC\otimes\cO_X$
  and $\cR_\Lambda$ arising as the middle term in the sequence
  \eqref{EulerSequenceGen} for $V=\Lambda_\bC$, so that we have the sequence
\begin{eq}\label{EulerSequenceLambda}
\begin{array}{ccccccccc}
0&\lra&\Omega_X&\lra&\cR_\Lambda&\lra&\Lambda_X&\lra&0
\end{array}
\end{eq}%
\item the algebraic torus $\bT_\Lambda=\Hom_{alg}(\Lambda,\bC^*)$.
\end{itemize}
The torus $\bT_\Lambda$ acts
on the graded ring of rational functions
$\cS_\Lambda = \bigoplus_{\lambda\in\Lambda} S^\lambda$ where
$$S^\lambda=\Gamma(X,\cO_X(D_\lambda)) = \{f\in\bC(X)^*:
\pd(f)+D_\lambda\geq 0\}\cup\{0\}$$ and the multiplication is as in
the field of rational functions $\bC(X)$.  We call $\cS_\Lambda$ the
Cox ring of the lattice $\Lambda$.

We note that the graded pieces $S^\lambda$ may be zero for some
$\lambda\in\Lambda$. In fact, we define a cone of effective divisors
$\Eff_\Lambda \subset \Lambda_\bR$ which is the closure of the cone
spanned by $\lambda$'s with $S^\lambda\ne 0$.  The cone $\Eff_\Lambda$
is convex and, because $X$ is projective, it is pointed which means
that it contains no non-trivial linear space. Equivalently, there exists
a linear form $\kappa: \Lambda\ra\bZ$ such that
$\kappa_{|\Eff_\Lambda}\geq 0$ and $\kappa(\lambda)\geq 1$ for every
non-zero $\lambda\in\Eff_\Lambda\cap\Lambda$ by \cite[Proposition~1.3]{BFJ}.
We define $\cS_\Lambda^+ = \bigoplus_{\kappa(\lambda)>0} S^\lambda$
which is a maximal ideal in $\cS_\Lambda$.

\par\medskip

Given a (quasi)coherent sheaf $\cF$ of $\cO_X$-modules by
$\cF(D_\lambda)$, we denote its twist $\cF\otimes
\cO(D_{\lambda})$. Next we define $\Lambda$-graded
$\cS_\Lambda$-module of sections
\begin{eq}\label{ModuleOfSections}
  \Gamma_\Lambda(\cF) = \bigoplus_{\lambda\in\Lambda}
  \H^0\left(X,\cF(D_\lambda)\right)
\end{eq}%
We will skip the subscript and write $\Gamma(\cF)$ if the choice of
$\Lambda$ is clear from the context.

The multiplication by elements of $\cS_\Lambda$ is defined by the
standard isomorphism $\cF(D_{\lambda_1})\otimes\cO(D_{\lambda_2}) \iso
\cF(D_{\lambda_1+\lambda_2})$, see \cite[Sect.~II.5]{Hartshorne}.
Clearly, $\Gamma_\Lambda$ extends to a functor from category of
(quasi) coherent sheaves on $X$ to category of graded $\cS_\Lambda$
modules.  The functor $\Gamma_\Lambda$ is left exact.

For the proof of Theorem \ref{GenJaczewskiThm} we need the following
observation.

\begin{lemma}\label{Rank1CoxModule}
  In the above situation, let $D$ be a divisor on $X$, which is linearly
  equivalent to $D_{\lambda_0}$ for some $\lambda_0\in\Lambda$. Then
  $\Gamma_\Lambda(\cO(D))$ is a free $\cS_\Lambda$-module of rank $1$
  which is isomorphic to $\cS_\Lambda[\lambda_0]$, where the square
  bracket denotes the shift in grading.
\end{lemma}


\subsection{Characteristic spaces of Cox rings}%
\label{CharacteristicSpacesSection}
The theory of Cox rings, or total coordinate rings, is pretty well
understood, see for example \cite{CoxJAG}, \cite{CoxTohoku},
\cite{HuKeel}, \cite{BerchtoldHausen} or \cite{ArzhantsevHausen}. The
reader not familiar with this subject may want to consult
\cite{LafaceVelasco} for a concise review or \cite{ADHL} for an
exhaustive overview of the subject. We note that in the present paper
the starting set-up is somehow more general. That is, the rings are
associated to a choice of $\Lambda$, hence our set up is similar to
this of Zariski, \cite[Sect.~4-5]{Zariski}. Thus, we will say that the
ring $\cS_\Lambda$ is the total coordinate ring of $X$ if and only if
the lattice $\Lambda$ is naturally isomorphic with $\Cl X$, the divisor
class group of $X$.
\par\medskip
The following condition for the pair $(X,\Lambda)$ is very convenient:
\begin{assumption}\label{ampleness-finitness}
Ampleness of $\Lambda$ and finite generation of
  $\cS_\Lambda$:
\begin{enumerate}
\item $\Lambda$ contains an ample divisor,
\item the ring $\cS_\Lambda$ is a finitely generated $\bC$-algebra.
\end{enumerate}
\end{assumption}

We note that the above condition is true if either $\Lambda$ is
generated by an ample divisor or $X$ is a Mori Dream Space (for
example, a $\bQ-$factorial Fano variety, see \cite[1.3.2]{BCHM}) and
$\Lambda$ generates the Picard group of $X$.

Under the above assumptions $Y_\Lambda=\Spec\cS_\Lambda$ is a well
defined affine variety over $\bC$ with the action of the torus
$\bT_\Lambda$. Moreover, the ideal $\cS^+_\Lambda$ defines the unique
$\bT_\Lambda$ invariant closed point in $Y_\Lambda$.

If $\cS_\Lambda$ is the total coordinate ring, then $Y_\Lambda$ is
normal, \cite[Sect.~I.5]{ADHL}

The relation of $Y_\Lambda$ with $X$ can be understood via the GIT
theory, see e.g.~\cite[Ch.~2]{HuKeel}. A choice of an ample divisor
$D_a\in\Lambda$ determines a character of $\bT_\Lambda$ and thus the
set of semistable points of this action $\widehat{Y}_\Lambda\subseteq
Y_\Lambda$, such that $Y_\Lambda\setminus \widehat{Y}_\Lambda$ is of
codimension 2 at least. Moreover $X$ is a geometric quotient of $\widehat{Y}_\Lambda$ by
$\pi_\Lambda: \widehat{Y}_\Lambda\lra X$.

A slightly different view is presented in \cite[Ch.~I]{ADHL} and we
will follow this approach. As in \cite[I.6.1]{ADHL}, the variety
$\widehat{Y}_\Lambda$, now called the characteristic space associated
to $\cS_\Lambda$, is constructed as the relative spectrum
$\Spec_X\widehat{\cS}_\Lambda$ of the graded sheaf of finitely
generated $\cO_X$-algebras $\widehat{\cS}_\Lambda =
\bigoplus_{\lambda\in\Lambda} \widehat{S}^\lambda$ where
$\widehat{S}^\lambda = \cO_X(D_\lambda)$. Following \cite{ADHL},
the sheaf $\widehat{\cS}_\Lambda$ will be called the Cox sheaf of the
lattice $\Lambda$ and $\cS_\Lambda$ is the algebra of its global
sections. The evaluation of global sections yields a map $\iota:
\widehat{Y}_\Lambda\lra Y_\Lambda$ which is an embedding onto an open
subset whose complement is of codimension $\geq 2$, see
\cite[I.6.3]{ADHL}. On the other hand the inclusion
$\cO_X=\widehat{S}^0\hookrightarrow \widehat{\cS}_\Lambda$ yield the
map $\pi:\widehat{Y}_\Lambda\lra X$ which is a geometric quotient of
the $\bT_\Lambda$ action. In fact, since $\Lambda$ consists of locally
principal divisors, the map $\pi$ is a locally trivial, principal
$\bT_\Lambda$-bundle, \cite[I.3.2.7]{ADHL}.

The ideal defining the closed set $Y_\Lambda\setminus
\widehat{Y}_\Lambda$ is called the irrelevant ideal and it is the
radical of an ideal generated by sections of a very ample line bundle
on $X$, \cite[I.6.3]{ADHL}.

\subsection{Cox rings of toric varieties}\label{CoxToricSection}
The following theorem was partly established by Cox (``only if'' part,
\cite{CoxJAG}), Hu and Keel (``if'' part, smooth case
\cite[Cor.~2.9]{HuKeel}) and, eventually, by Berchtold and Hausen,
\cite[Cor.~4.4]{BerchtoldHausen} who proved it in a stronger form than
the one below. We recall that a variety $X$ satisfies W{\l}odarczyk's
$A_2$ property if any two points of $X$ are contained in common open
affine neighbourhood, see \cite{Wlodarczyk}.
\begin{theorem}\label{CoxHuKeel}
  Let $X$ be a complete normal variety which has $A_2$ property and
  has finitely generated free divisor class
  group. 
  Then $X$ is a toric variety if and only if its total coordinate ring
  is a polynomial ring.
\end{theorem}

We will need a slight variation of the above result.
\begin{theorem}\label{CoxHuKeel2}
  Let $X$ be a projective normal variety with with a lattice
  $\Lambda\subset\Pic X$ containing a class of an ample divisor. If
  $\cS_\Lambda$ is a polynomial ring, then $X$ is a smooth toric
  variety and $\Lambda=\Pic X = \Cl X$.
\end{theorem}

\begin{proof}
  The proof of the first part, that is of toricness of $X$, goes along
  the lines of the proof of \cite[2.10]{HuKeel}. The references
  regarding the group action were suggested to us by J\"urgen Hausen.

  As usually, by $n$ we denote the dimension of $X$ and by $r$ the
  rank of the lattice $\Lambda$. Thus, in the situation of the
  theorem, $\cS_\Lambda$ is the polynomial ring in $n+r$ variables.
  By the linearization theorem \cite[5.1]{KraftPopov} the action of
  $\bT_\Lambda$ on $Y=\Spec\cS_\Lambda=\bA^{r+n}$ is linear. By
  \cite[2.4]{Swiecicka} $\widehat{Y}$ is a toric variety and the
  quotient of $\widehat{Y}$ by the torus action, which is $X$, is a
  toric variety again. The morphism $\widehat{Y}\ra X$ is surjective
  and toric, therefore the set the rays in the fan defining
  $\widehat{Y}$ maps surjectively onto the set of rays in the fan of
  $X$. Thus, the rank of $\Cl X$, which is equal to the number of rays
  in the fan of $X$ minus its dimension, is equal $r$ at most. On the
  other hand, since $\Lambda\subset\Pic X\subset\Cl X$, it follows
  that $r=\rank\Lambda=\rank\Cl X$ and the maps of fans of
  $\widehat{Y}$ and $X$ induces bijection on the set of their rays.

  By fixing a basis in $\Cl(X)$ we can see the Cox ring of $\Lambda$
  as a subring of the total coordinate ring of $X$; both rings are
  polynomial rings in $n+r$ variables. The inclusion of rings can be
  explained in terms of their lattices of Laurent monomials in the following
  way.

  Let $\widehat{M} =\bigoplus_{i=1}^{r+n} \bZ e_i^*$ denote the
  lattice of characters of the natural torus in $Y$, where $e_i^*$'s
  are dual to the primitive generators of the rays of $\widehat{Y}$
  and let $M$ denote the lattice of characters of the quotient torus
  acting on $X$. Let $\pi:\widehat{M}^\vee \ra M^\vee$ be the morphism
  of lattices, corresponding to the toric morphism $\widehat{Y}\ra X,$
  where $\pi(e_i)=a_i v_i$ for a positive integer $a_i$ and the
  primite vector $v_i\in M^\vee$, spanning a ray of the fan of
  $X$. The dual of the map $\pi$ fits into an upper exact sequence in
  the diagram \ref{Lattices} below.

  The vertical arrows in this diagram come from the inclusion of
  $\cS_\Lambda$ in the total ring of $X$. Namely, given an element in
  $\cS_\Lambda$ which is a monomial in the polynomial representation
  of this ring, the associated effective divisor in the total
  coordinate ring of $X$ will be $\bT_M$-invariant. This follows
  because the linearization of the action of $\bT_M$ on the graded
  pieces of $\cS_\Lambda$ agrees with the inclusion $M\hookrightarrow
  \widehat{M}$. Therefore we get a map $\iota$ in the center of the
  following commutative diagram in which $D_i$' denote $\bT_M$
  invariant divisors, each $D_i$ associated to the ray spanned by
  $v_i$.
  \begin{eq}\label{Lattices}
    \xymatrix{ 0\ar[r] & M\ar[r]^{\pi^*}\ar@{=}[d] &
      \widehat{{\phantom{|}}M{\phantom{|}}}\ar[r]\ar@{^{(}->}@<-2pt>[d]^{\iota}
      &
      {\phantom{|}}\Lambda{\phantom{|}}\ar[r]\ar@{^{(}->}@<-2pt>[d(0.8)] & 0 \\
      0\ar[r] & M\ar[r] & \bigoplus_{i=1}^{r+n}\bZ D_i\ar[r] &
      \Cl(X)\ar[r] & 0 }
  \end{eq}
  Thus the lattice $\widehat{M}$ is of finite index in the lattice
  $\bigoplus_{i=1}^{r+n}\bZ D_i$ and, in addition, over $\bQ$ the
  simplicial cones of effective divisors in $\widehat{M}$ and
  $\bigoplus_{i=1}^{r+n}\bZ D_i$ are equal. That is, $\iota
  (e_i^*)=b_{\sigma(i)} D_{\sigma(i)},$ where $\sigma$ is some
  permutation of the set $\{1,\ldots,r+n\}$ and $b_i$'s are positive
  integers. By the commutativity of the left hand side of
  \eqref{Lattices} we get that $\iota(\sum a_i \langle m, v_i\rangle
  e_i^*)= \sum\langle m, v_i\rangle D_i$, for every $m\in
  M$. Consequently we get that $a_j b_{\sigma(j)}v_j=v_{\sigma(j)}$,
  which implies $\sigma=id$ and $a_i=b_i=1$ because the rays are
  different and $a_i, b_i$ are positive integers. Hence $\Lambda=\Cl
  X$, equivalently $\cS_\Lambda$ is the total coordinate ring of $X$.

  Alternatively, the latter statement follows from much
  more general results \cite[6.4.3, 6.4.4]{ADHL}. Thus $\Pic X=\Cl X$
  which yields that $X$ is smooth, see \cite[4.2.6]{CoxLittleSchenck}.
\end{proof}

\par\bigskip
Although Theorems \ref{Jaczewski} and \ref{CoxHuKeel} give
characterisation of toricness in terms of apparently different objects
they turn out to be closely related. We will discuss this issue in the
second part of the paper.


\subsection{Differential forms and splitting of tangent bundle}%
\label{DiffFormsSection}
We will use the standard definition of the module and the sheaf of
K\"ahler $\bC$ differentials, as in \cite[Sect.~26]{Matsumura}, which
we will denote by $\Omega_A$ and $\Omega_X$, respectively. The double
dual, $\Omega^{\vee\vee}_A$ or $\Omega^{\vee\vee}_X$ will be called,
respectively, the module, or the sheaf of Zariski or reflexive
differentials, see e.g.~\cite[Sect.~1]{Knighten}. We note that, being
reflexive, over normal varieties the sheaf of Zariski differentials
has extension property which means that all its sections are
determined uniquely on complements of codimension $\geq 2$ sets,
\cite{Hartshorne2}. That is, if $Y$ is normal and
$\iota:\widehat{Y}\hookrightarrow Y$ is an open subset such that
$Y\setminus\widehat{Y}$ is of codimension $\geq 2$ then
$\iota_*\Omega^{\vee\vee}_{\widehat{Y}}=\Omega^{\vee\vee}_Y$.

\medskip

Let us note that the theorem of Jaczewski or, more generally, any
theorem about splitting of the sheaf $\cR_\Lambda$, see
\ref{GenJaczewskiThm}, can be considered in the context of splitting
of the sheaf of K\"ahler differentials, $\Omega_X$ (just take
$\Lambda=0$). This question for complex K\"ahler manifolds has been
considered by several authors: \cite{Beauville}, \cite{Druel},
\cite{CampanaPeternell}, \cite{BPT} and \cite{Horing1},
\cite{Horing2}. For non-K\"ahler manifolds, a Hopf manifold of
dimension $\geq 2$, provides an example of a complex manifold whose
tangent bundle (the dual of the sheaf of K\"ahler differentials)
splits into a sum of line bundles, see \cite[Ex.~2.1]{Beauville}.


\section{Part II, results}

\subsection{The generalized Euler extension}\label{GenEulerSection}

In the present section we want to describe the generalized Euler
sequence \eqref{EulerSequenceGen} in terms of \v{C}ech data. We fix an
affine covering $\cU=(U_i)_{i\in I}$ of $X$ in which all divisors in
$\Lambda$ can be represented as principal divisors.  That is, for
every $D_\lambda\in\Lambda\subset \CDiv(X)$ and $i\in I$ we can choose
a rational function $f_i^\lambda\in\bC(X)^*$ such that
$D_{\lambda|U_i}= \pd(f_i^\lambda)$. We define also $g_{ij}^\lambda =
f_j^\lambda/f_i^\lambda \in \cO^*_X(U_i\cap U_j)$. The first Chern
class $c_1(\cO(D))$ is represented (up to the multiplicative constant)
in the covering $\cU$ by the {\v C}ech cocycle $(d\log g_{ij}^\lambda)
\in \cZ^1(\cU,\Omega_X)$. The choice of $f_i$'s and hence of
$g_{ij}$'s is unique up to functions from $\cO^*_X(U_i)$.

  Thus, in order to avoid ambiguity we fix a basis of the lattice
  $\Lambda$ and $f_i$'s for the basis elements. This determines the
  choice $f_i^\lambda$ for every $D_\lambda\in\Lambda$ so that we have
  a group homomorphism of $\Lambda$ into \v Cech 0-cochains of
  rational functions
  \begin{eq}\label{HomLatticeToChain}
    \Lambda\ni\lambda \mapsto (f_i^\lambda) \in \cC^0(\cU,\bC(X)^*)
  \end{eq}%
  As the result, we get a homomorphism into \v Cech cocycles
  \begin{eq}\label{HomLatticeToCycle}
    \Lambda\ni\lambda \mapsto (g_{ij}^\lambda) \in \cZ^1(\cU,\cO^*_X)
  \end{eq}

  The map $c_1: \Lambda\lra \H^1(X,\Omega_X)$ is represented in the
  covering $\cU$ by $\psi_{ij} \in \Hom(\Lambda,\Omega_X(U_i\cap
  U_j))$ such that
  $$\psi_{ij}(\lambda)=\psi_{ij}(D_\lambda)=
  d\log g_{ij}^\lambda=d\log f_j^\lambda-d\log f_i^\lambda$$

  Sections of the sheaf $\cR_\Lambda$ over an open $U\subset X$ come
  by glueing $(\omega_i,\lambda_i) \in \Omega_X(U_i\cap U) \oplus
  \Lambda_X(U_i\cap U)$, with $(\omega_j,\lambda_j)\in\Omega_X(U_j\cap
  U)\oplus\Lambda_X(U_j\cap U)$ and
  $\omega_j=\omega_i+\widetilde{\psi_{ij}}(\lambda_i)$, where the
  restriction of $\omega$'s and $\lambda$'s to $U_i\cap U_j\cap U$ is
  denoted by the same letter, c.f.~\cite[Sect.~4]{Atiyah}, and
  $\widetilde{\psi_{ij}}$ is an extension of $\psi_{ij}$ to a
  homomorphism $\Lambda_X(U\cap U_i\cap U_j)\ra \Omega_X(U\cap U_i\cap
  U_j)$.


\subsection{Differentials of the Cox sheaf}\label{DiffCoxSection}
  Now, for a nonzero rational function $h\in S^\lambda=
  \H^0(X,\cO_X(D_\lambda))$ we consider its local regular
  presentation in covering $\cU$ given by
  \begin{eq}\label{CyclePresentation}
    s_i=hf^\lambda_i \in \cO_X(U_i)
  \end{eq}%
  We will write $s:=(s_i)_{i\in I}$. We claim that the pairs
  $(ds_i,s_i\cdot D_\lambda)\in \Omega_X(U_i)\oplus \Lambda_X(U_i)$
  determine a section of $\cR_\Lambda(D_\lambda)=\cR_\Lambda\otimes
  \cO(D_\Lambda)$ which we call $ds$. Here we use the convention that
  $(ds_i,s_i\cdot D_\lambda)$ is identified with $(ds_i\otimes 1,
  D_\lambda\otimes s_i)$ as a section of $\Omega_X(D_\lambda)\oplus
  \Lambda_X(D_\lambda)$ over $U_i$.  Indeed, since over the set
  $U_i\cap U_j$ we have $s_j=g_{ij}^\lambda s_i$, it follows that
  \begin{eq}\label{LeibnizIdentity}
    ds_j=g_{ij}^\lambda\cdot ds_i + s_i\cdot d g_{ij}^\lambda =
    g_{ij}^\lambda \cdot (ds_i + s_i\cdot d \log g_{ij}^\lambda )=
    g_{ij}^\lambda \cdot (ds_i + s_i\cdot\psi_{ij}(D_\lambda)
    ) \end{eq}
  hence we get the statement. The map $s\mapsto ds$ is $\bC$ linear
  and it satisfies the Leibniz rule. Indeed, for $s\in S^\lambda$ and
  $s'\in S^{\lambda'}$, we verify that
  $$d(s\cdot s')=(d(s_i\cdot s_i',s_i\cdot s_i'\cdot
  (D_\lambda+D_{\lambda'}))= s_i(ds_i',s_i'\cdot D_{\lambda'}) +
  s_i'(ds_i,s_i\cdot D_{\lambda})=s\cdot ds'+s'\cdot ds$$ Thus the map
  \begin{eq}\label{Derivation}
    \cS_\Lambda\supset S^\lambda=\H^0(X,\cO_X(D_\lambda))\ni s\lra
    ds=(ds_i,s_i\cdot D_\lambda)\in
    \H^0(X,\cR_\Lambda(D_\lambda))\subset \Gamma(\cR_\Lambda)
  \end{eq}%
  is a $\bC$-derivation of the Cox ring $\cS_\Lambda$.

  It is clear that the above construction of the map $s\ra ds$ works
  also locally, for sections over an arbitrary open subset of $X$,
  hence it gives a map $d: \cO_X(D_\lambda)\ra
  \cR(D_\lambda)$. Actually, the $\bC$-linear map $d$ can be described
  in relation to the jet bundles, which were presented in
  Section \ref{JetSection}. The first Chern class of $D_\lambda$ gives
   the following
  diagram with exact rows of left-to-right homomorphisms of sheaves of
  $\cO_X$-modules
  \begin{eq}\label{JetCoxDiagram}
    \xymatrix{
    0\ar[r]&\Omega_X(D_\lambda)\ar[r]\ar@{=}[d]&\cR_{D_\lambda}(D_\lambda)
    \ar[r]\ar[d]&
    \cO_X(D_\lambda)\ar@/_/[l]\ar@/_/[dl]^d\ar[r]
    \ar[d]^{c_1(D_\lambda)\otimes id}
    &0\\
    0\ar[r]&\Omega_X(D_\lambda)\ar[r]&\cR_\Lambda(D_\lambda)\ar[r]&
    \Lambda\otimes\cO_X(D_\lambda)\ar[r]&0 }\end{eq}
  As discussed in Section \ref{JetSection} the upper row of this
  sequence splits as a sequence of $\bC$-modules and the map $d$ is
  the composition of the splitting map with the map of extensions.

  The map $d$ is a $\bC$-derivation of the sheaf
  $\widehat{S}_\Lambda$. That is, for every open $U\subseteq X$, any
  section $s\in \widehat{S}^\lambda(U)$ yields
  $ds\in(\cR_\Lambda\otimes\widehat{S}^\lambda)(U)$ and the map
  $s\mapsto ds$ is a derivation. Verification of the Leibniz identity
  is the same as in \eqref{LeibnizIdentity}. Thus, we get the
  homomorphism of $\Lambda$-graded quasi-coherent sheaves of
  $\cO_X$-modules $\widehat\varphi: \pi_*\Omega_{\widehat
    Y}\lra\cR_\Lambda\otimes_{\cO_X}\widehat{\cS}_\Lambda$.

\begin{lemma}\label{DiagramLemma} Suppose that we are in the set-up
  of section \ref{CharacteristicSpacesSection} so that, in particular,
  the pair $(X,\Lambda)$ satisfies the assumption
  \ref{ampleness-finitness} and the characteristic space
  $\widehat{Y}_\Lambda$ with the projection
  $\pi:\widehat{Y}_\Lambda\ra X$ is well defined. Then the map
  $\widehat\varphi$ fits into the following commutative diagram of
  $\Lambda$ graded sheaves of $\cO_X$-modules whose rows are exact and
  vertical arrows are isomorphisms:
  \begin{eq}\label{DiffCoxPfDiagram}
  \xymatrix{ 0\ar[r]&\pi_*\pi^*\Omega_X\ar[r]\ar[d]&
    \pi_*\Omega_{{\widehat Y}_\Lambda}\ar[r]\ar[d]^{\widehat\varphi}&
    \pi_*\Omega_{{\widehat Y}_\Lambda/X}\ar[r]\ar[d]& 0\\
    0\ar[r]&\Omega_X \otimes_{\cO_X}{\widehat\cS}_\Lambda \ar[r]&
    \cR_\Lambda\otimes_{\cO_X}{\widehat\cS}_\Lambda\ar[r]&
    \Lambda_X\otimes_{\cO_X}{\widehat\cS}_\Lambda\ar[r]&0 }
  \end{eq}%
  The lower row of the above diagram comes by extending the
  coefficients in \eqref{EulerSequenceGen} while the upper sequence is
  the relative cotangent sequence for $\pi$.
\end{lemma}
\begin{proof}
  Let us recall that $\pi:\widehat Y\lra X$ is a locally trivial
  principal $\bT_\Lambda$-bundle. That is, for $U_i\subseteq X$
  trivializing all divisors in $\Lambda$ there is an identification of
  $\cO_X(U_i)$-modules $\widehat{S}^\lambda(U_i) = (f_i^\lambda)^{-1}
  \cdot \cO_X(U_i)$, as submodules of $\bC(X)(U_i)$. Thus, having in
  mind homomorphism \eqref{HomLatticeToChain}, we get isomorphism of
  $\cO_X(U_i)$-algebras $\widehat{\cS}_\Lambda(U_i) \iso
  \cO_X(U_i)[\Lambda]=\cO_X(U_i)\otimes\bC[\Lambda]$. Therefore, by
  the formula for tensor products \cite[16.5]{Eisenbud}, we get the
  canonical isomorphism
  $$\Omega_{\widehat{\cS}_\Lambda(U_i)}=
  \Omega_X(U_i)\otimes_{\cO_X(U_i)}\widehat{\cS}_\Lambda(U_i)\ \oplus\
  \Omega_{\bC[\Lambda]}\otimes_{\bC[\Lambda]}
  \widehat{\cS}_\Lambda(U_i)$$ which splits the upper row over
  $U_i$. This splitting coincides with the splitting of the lower row
  over $U_i$ once we note the standard isomorphism
  $\Omega_{\bC[\Lambda]} = \bC[\Lambda] \otimes_\bZ\Lambda$, which
  sends the derivative of the monomial $t^\lambda\in\bC[\Lambda]$ to
  $t^\lambda\otimes \lambda$, see
  \cite[Ch.~8]{CoxLittleSchenck}. Moreover, under this identification
  the arrows in the upper row are identical with those in the lower
  row, c.f.~\eqref{CyclePresentation}. Thus the diagram commutes and
  all vertical arrows are now defined globally which concludes the
  proof of Lemma \ref{DiagramLemma}.
\end{proof}

The above lemma is a generalization of \cite[2.1]{KPSW}. In
particular we get the following corollary

\begin{corollary}\label{GenEulerT-inv}
  Suppose that we are in the situation of Lemma
  \ref{DiagramLemma}. Then the generalized Euler sequence
  \eqref{EulerSequenceGen} associated to $\Lambda_\bC$ is the
  $\bT_\Lambda$-invariant part (zero grading with respect to
  $\Lambda$) of the exact sequence of sheaves of differentials
  associated to the map $\pi$:
  \begin{eq}\label{DiffCoxThmDiagram}
    0\lra \pi^*\Omega_X\lra \Omega_{\widehat Y_\Lambda} \lra
    \Omega_{\widehat Y_\Lambda/X}\lra 0
  \end{eq}
\end{corollary}

The next corollary follows immediately by the extension property of
Zariski differentials, which was explained in Section
\ref{DiffFormsSection}. We recall that the Cox ring $\cS_\Lambda$ is integrally closed by~\cite[Theorem~1.1]{EKW}.

\begin{corollary}\label{DiffCoxRing}
  Suppose that we are in the situation of Lemma \ref{DiagramLemma} and
  assume that the sheaf of K\"ahler differentials on $X$ is
  reflexive. Then the $\cS_\Lambda$-module $\Gamma(\cR_\Lambda)$ (see
  \ref{ModuleOfSections}) is isomorphic to the module of Zariski
  differentials of the Cox ring $\cS_\Lambda$:
  $$\Gamma(\cR_\Lambda)\iso\Omega_{\cS_\Lambda}^{\vee\vee}$$
\end{corollary}


Clearly, the sheaf $\cR_\Lambda$ will not change if $\Lambda$ is
replaced by another lattice whose $\bC$ linear span is the same as
$\Lambda$. That is, if $\Lambda'\subset\Lambda$ is a sublattice of
finite index then $\cR_{\Lambda'}=\cR_\Lambda$ and thus
$\Gamma_{\Lambda'}(\cR_{\Lambda'})$ is $\Lambda'$ graded part of
$\Gamma_{\Lambda}(\cR_{\Lambda})$. Thus, if we are in the situation of
Corollary \ref{DiffCoxRing} this gives a clear description of Zariski
differentials on $\Spec \cS_{\Lambda'}$ in terms of Zariski
differentials over $\Spec \cS_{\Lambda}$. On the other hand, the
behaviour of K\"ahler differentials is much more intricate.  This
problem was studied in a special situation of $X=\bP^n$ in
\cite{GrebRollenske}.

\subsection{Generation of the Cox ring}\label{FiniteGenerationSection}
Let us apply the functor $\Gamma$ defined in Section
\ref{CoxRingSection} to generalized Euler sequence
\eqref{EulerSequenceLambda}. The result if the following sequence of
$\Lambda$-graded $\cS_\Lambda$-modules
$$\begin{array}{ccccccc}
 0&\lra&\Gamma(\Omega_X)&\lra&
\Gamma(\cR_\Lambda)&\lra&\Gamma(\Lambda_X)=\Lambda\otimes_\bZ\cS_\Lambda
\end{array}
$$
We compose the right hand arrow in the above sequence with
$\kappa\otimes id: \Lambda\otimes\cS_\Lambda\ra\cS_\Lambda$ where
$\kappa$ is the map defined in Section \ref{CoxRingSection}.  We call
the resulting homomorphism
$\widehat{\kappa}:\Gamma(\cR_\Lambda)\ra\cS_\Lambda$.
\begin{lemma}\label{kappaSurjective}
  The homomorphism $\widehat{\kappa}$ is surjective onto
  $\cS_\Lambda^+$.
\end{lemma}
\begin{proof}
  First we note that the image of $\widehat{\kappa}$ is contained in
  $\cS^+_\Lambda$. This is because the map $\H^0(X,\cR_\Lambda) \ra
  \H^0(X,\Lambda_X)$ is zero. Next we use \eqref{JetCoxDiagram} to get
  the following commutative diagram
\begin{eq}\label{kappa}
  \xymatrix{ \H^0(X,\cR_{D_\lambda}(D_\lambda))\ar[r]\ar[d]&
    \H^0(X,\cR_\Lambda(D_\lambda))\ar[d]\ar[dr]^{\widehat{\kappa}}& \\
    \H^0(X,\cO(D_\lambda))\ar[r]\ar[r]_{c_1(D_\lambda)\otimes
      id\ \ \ }\ar@/_/[u]&
    \Lambda\otimes\H^0(X,\cO(D_\lambda))\ar[r]_{\ \ \ \kappa\otimes id}&
    \H^0(X,\cO(D_\lambda))
      }
\end{eq}

The composition of the homomorphisms in the lower row is
an isomorphism because for any effective non-zero divisor $\kappa(c_1(D_\lambda))\ne 0$.
Hence $\widehat{\kappa}$ is onto $\cS_\Lambda^+$.
\end{proof}

We recall the following observation, which is classical and probably
know since Hilbert's time.
\begin{lemma}\label{littleHilbert}
  Let $\cA=\bigoplus_{m\geq 0}\cA^m$ be $\bZ_{\geq 0}$-graded ring. If
  homogeneous elements $a_1,\dots, a_t$ generate the ideal
  $\cA^+=\bigoplus_{m >0}\cA^m$, then they generate $\cA$ as
  $\cA^0$-algebra.
\end{lemma}

Combining the two above lemmata we get

\begin{corollary}\label{generation}
  If $\Gamma(\cR_\Lambda)$ is generated as $\cS_\Lambda$-module by
  $\Lambda$-homogeneous elements $m_1,\dots,m_t$, then $\cS_\Lambda$ is
  a finitely generated $\bC$-algebra with generators
  $\widehat{\kappa}(m_1),\dots, \widehat{\kappa}(m_t)$.
\end{corollary}

The above result can be put as a part of the following equivalence
statement.

\begin{theorem}\label{GenerationEquivalence}
  Let $X$ be a projective, normal variety and $\Lambda\subset \Pic X$ a
  finitely generated lattice of Cartier divisors as in Section \ref{CoxRingSection}.
  Suppose moreover that $\Lambda$ contains an ample divisor and the
  sheaf $\Omega_X$ is reflexive. Then the following conditions are
  equivalent:
\begin{enumerate}
\item The Cox ring $\cS_\Lambda$ is a finitely generated $\bC$-algebra.
\item The module $\Gamma_\Lambda(\cR_\Lambda)$ is finitely generated
  over $\cS_\Lambda$.
\item For every reflexive sheaf $\cF$ over $X$, the graded module
  $\Gamma_\Lambda(\cF)$ is finitely generated over $\cS_\Lambda$.
\end{enumerate}
\end{theorem}

\begin{proof}
  We proved the implication $(2)\Rightarrow(1)$. The implication
  $(3)\Rightarrow(2)$ is clear since $\cR_\Lambda$ is reflexive.  To
  prove the implication $(1)\Rightarrow(3)$ suppose that $\cF$ is
  reflexive and $D_{\lambda_0}\in\Lambda$ is ample. Thus
  $\cF^\vee(D_{m\lambda_0})$ is generated by global section for $m\gg
  0$ or, equivalently, we have a surjective morphism $\cO^{\oplus
    N}\ra\cF^\vee(D_{m\lambda_0})$ for some positive integer
  $N$. Dualising and twisting we get injective
  $\cF\ra\cO(D_{m\lambda_0})^{\oplus N}$. By left-exactness of
  $\Gamma_\Lambda$ we can present $\Gamma_\Lambda(\cF)$ as a submodule
  of the free finitely generated $\cS_\Lambda$-module
  $\cS_\Lambda[m\lambda_0]^{\oplus N}$. Because $\cS_\Lambda$ is
  Noetherian $\Gamma_\Lambda(\cF)$ is finitely generated too.
\end{proof}

\subsection{The theorem of Jaczewski revisited}\label{GenJaczewskiThmSection}

\begin{theorem}\label{GenJaczewskiThm} Suppose that $X$ is a smooth
  projective variety and $\Lambda$ a free finitely generated group of
  Cartier divisors, as defined in the section \ref{CoxRingSection},
  which contains an ample divisor. If $\cR_\Lambda$ splits into the
  sum of line bundles, that is $\cR_\Lambda=\bigoplus\cL_i$, and for
  every $\cL_i$ we have $\cL_i\iso\cO(-D_{\lambda_i})$ for some
  $\lambda_i\in\Lambda$, then $X$ is a toric variety, $\Lambda=\Pic X$
  and $D_i$'s are linearly equivalent to torus invariant prime divisors.
\end{theorem}

\begin{proof}
  By Lemma \ref{Rank1CoxModule} the module $\Gamma(\cR_\Lambda)$ is free of
  rank $n+r$.  Therefore, by Corollary \ref{generation} the algebra
  $\cS_\Lambda$ is generated by $n+r$ elements.  Since its dimension
  is also equal to $n+r$ it is the polynomial ring and $X$ is a toric
  variety and $\Lambda=\Pic X$ by Theorem \ref{CoxHuKeel2}. On the
  toric variety $X$ the Euler sheaf $R_{\Pic X}$ splits as a sum of line bundles
  associated to the negatives of all torus invariant prime divisors (c.f.~\cite{BatyrevMelnikov}
  or Colloray~\ref{GenEulerT-inv}).
  The theorem follows since such decomposition is essentially unique by~\cite{AtiyahKrullSchmidt}.
\end{proof}

We can reformulate this result in the spirit of the original Jaczewski
theorem, \ref{Jaczewski}.

\begin{corollary}\label{GenJaczewski}
  Let $X$ be a smooth projective variety with $\H^1(X,\cO_X)=H^2(X,\cO_X)=0$.
  If the potential sheaf of $X$ splits into the direct sum of line
  bundles, then $X$ is a toric variety.
\end{corollary}
\begin{proof}
  We note that, by assumptions, it follows that $\Pic
  X=\H^2(X,\bZ)$ is a finitely generated abelian group.

  First, let us assume that $\Pic X$ has no torsion. In this case we
  let $\Lambda=\Pic X$. Then, because of our assumptions,
  $\Lambda_\bC=\H^1(X,\Omega_X)$ and therefore $\cR_\Lambda=\cR_X$ is
  the dual of the potential sheaf.  Thus we are in the situation of
  \ref{GenJaczewskiThm} and the corollary follows.

  Now, we consider the general case when, possibly, the torsion part
  of $\Pic X$ is non-zero, we denote it by $G$ and we will show that
  it is impossible. The group $G$ is finite. We can lift $G$ to $\CDiv
  X$ and treat its elements as Cartier divisors. We define $G$-graded
  sheaf of $\cO_X$-algebras
  $$\widehat{\cS}_G=\bigoplus_{D\in  G}\cO_X(D)$$
  where, for every Zariski open $U\subset X$, we define
  $\cO_X(D)(U)=\{f\in\bC^*(X): \pd_U(f)+D_{|U}\geq 0\}\cup\{0\}$, and
  the multiplication is as in $\bC(X)$. As in~\cite{ADHL}, see also
  Section \ref{CharacteristicSpacesSection}, we define
  $\widehat{Y}_G=\Spec_{\cO_X}\widehat{\cS}_G$, the characteristic
  space of $G\subset\Pic X$. The resulting morphism
  $\pi:\widehat{Y}_G\ra X$ is \'etale with Galois group $G$. We note
  that this is a slight generalization of the well known construction
  of unramified covering, see e.g.~\cite[Section~I.17]{BPVdV}

  Now we use \cite[Lemma 2.2, Prop.~4.2]{KKV} to conclude that
$\Pic\widehat{Y}_G$ is torsion free. Since the morphism $\pi$ is \'etale and the definition
of the Euler sheaf is local (c.f. Section~\ref{GenEulerSection}), we conclude that
the Euler sheaf $\cR_\Lambda$, where $\Lambda=\Pic \widehat{Y}_G$, is a pullback
of the potential sheaf $\cR_X.$ By Theorem~\ref{GenJaczewskiThm} the variety $\widehat{Y}_G$ is
toric.

The finite group $G$ acts freely on $\widehat{Y}_G$, however, any element of $G$ has a fixed point on $\widehat{Y}_G$ by the
holomorphic Lefschetz fixed-point formula (see~\cite{AtiyahBott} or \cite{GriffithsHarris})
which is impossible unless $G$ is trivial.

\end{proof}

\def\polhk#1{\setbox0=\hbox{#1}{\ooalign{\hidewidth
  \lower1.5ex\hbox{`}\hidewidth\crcr\unhbox0}}}

\end{document}